\documentclass[11pt]{article}
\usepackage[left=1in,top=1in,right=1in,bottom=1in,head=.1in,nofoot]{geometry}

\usepackage{setspace,url,bm,amsmath} 

\usepackage{titlesec} 
\titlelabel{\thetitle.\quad} 
\titleformat*{\section}{\bf\large\center\uppercase} 

\usepackage{graphicx} 
\usepackage{bbm}
\usepackage{latexsym}
\usepackage{caption}
\usepackage[margin=20pt]{subcaption}
\usepackage{hyperref}

\usepackage[table]{xcolor}

\newcommand{\GG}[1]{}

\usepackage{amsthm}
\theoremstyle{definition}

\newtheorem{prop}{Proposition}
\newtheorem{lemma}{Lemma}

\newtheorem*{corollary*}{Corollary}

\usepackage{natbib} 
\bibpunct{(}{)}{;}{a}{}{,} 

\usepackage{etoolbox} 
\apptocmd{\sloppy}{\hbadness 10000\relax}{}{} 

\usepackage{color}
\usepackage{listings}

\DeclareMathOperator*{\argmax}{arg\,max}

\begin{document}
\doublespacing
\title{\bf Demystifying the bias from selective inference: A revisit to Dawid's treatment selection problem}
\author{Jiannan Lu \footnote{Address for correspondence: Jiannan Lu, Microsoft Corporation, One Microsoft Way, Redmond, Washington 98052, USA.
Email: \texttt{jiannl@microsoft.com}}~ and Alex Deng  \\ Analysis and Experimentation, Microsoft Corporation}
\date{}
\maketitle
\begin{abstract}
We extend the heuristic discussion in \cite{Senn:2008} on the bias from selective inference for the treatment selection problem \citep{Dawid:1994}, by deriving the closed-form expression for the selection bias. We illustrate the advantages of our theoretical results through numerical and simulated examples. 
\end{abstract}

\section{Introduction}
Selective inference gained popularity in recent years \citep[e.g.,][]{Lockhart:2014, GSell:2016, Reid:2016}. To quote \cite{Dawid:1994}, ``... a great deal of statistical practice involves, explicitly or implicitly, a two stage analysis of the data. At the first stage, the data are used to identify a particular parameter on which attention is to focus; the second stage then attempts to make inferences about the selected parameter.'' Consequently, the results (e.g., point estimates, $p-$values) produced by selective inference are generally ``cherry-picked'' \citep{Taylor:2015}, and therefore it is of great importance for practitioners to conduct ``exact post-selection inference'' \citep[e.g.,][]{Tibshirani:2014, Lee:2015}. 

To demonstrate the importance of ``exact post-selection inference,'' in this paper we focus on the ``bias'' of the posterior mean associated with the most extreme observation (formally defined later, and henceforth referred to as ``selection bias'') in the treatment selection problem \citep{Dawid:1994}, which is not only fundamental in theory, but also of great practical importance in, e.g., agricultural studies, clinical trials, and large-scale online experiments \citep{Kohavi:2013}. In an illuminating paper, \cite{Senn:2008} provided a heuristic explanation that the existence of selection bias depended on the prior distribution, and upheld Dawid's claim that the fact that selection bias did not exist in some standard cases was a consequence of using certain conjugate prior. In this paper, we relax the modeling assumptions in \cite{Senn:2008} and derive the closed-form expression for the selection bias. Consequently, our work can serve as a complement of the heuristic explanation provided by \cite{Senn:2008}, and is useful from both theoretical and practical perspectives.

The paper proceeds as follows. Section \ref{sec:intro} reviews the treatment selection problem, defines the selection bias, and describes the Bayesian inference framework which the remaining parts of the paper are based on. Section \ref{sec:theory} derives the closed-form expression for the selection bias. Section \ref{sec:examples} highlights numerical and simulated examples that illustrates the advantages of our theoretical results. Section \ref{sec:discuss} concludes and discusses future directions.

\section{Bayesian Inference for Treatment Selection Problem}\label{sec:intro}

\subsection{Treatment Selection Problem and Selection Bias}

Consider an experiment with $p \ge 2$ treatment arms. For $i=1, \ldots, p,$ let $\mu_i$ denote the mean yield of the $i$th treatment arm. After running the experiment, we observe the sample mean yield of the $i$th treatment arm, denoted as $X_i.$ Let
\begin{equation*}
i^* = \argmax_{1 \le i \le p} X_i    
\end{equation*}
denote the index of the largest observation. The focus of selective inference is on $\mu_{i^*},$ which relies on $X_1, \ldots, X_p.$ We let
$
\textrm{E} ( \mu_{i^*} \mid X_{i^*})
$
be the posterior mean of $\mu_{i^*}$ as if it were selected before the experiment, and \begin{equation*}
\textrm{E} ( \mu_{i^*} \mid X_{i^*}, X_{i^*} = \max X_i)
\end{equation*}
be the ``exact post-selection'' posterior mean of $\mu_{i^*},$ which takes the selection into account. Following \cite{Senn:2008}, we define the selection bias as
\begin{equation}\label{eq:target-original}
\Delta = \textrm{E} ( \mu_{i^*} \mid X_{i^*}) - \textrm{E} ( \mu_{i^*} \mid X_{i^*}, X_{i^*} = \max X_i ).
\end{equation}

Having defined the selection bias, we briefly discuss the ``selection paradox'' in \cite{Dawid:1994}, i.e., ``since Bayesian posterior distributions are already fully conditioned on the data, the posterior distribution of any quantity is the same, whether
it was chosen in advance or selected in the light of the data.'' In other words, if we define the selection bias as
\begin{equation*}
\tilde \Delta = \textrm{E} ( \mu_{i^*} \mid X_1, \ldots, X_p) - \textrm{E} ( \mu_{i^*} \mid X_1, \ldots, X_p, X_{i^*} = \max X_i ),
\end{equation*}
then indeed $\tilde \Delta = 0.$

\subsection{The Normal-Normal Model}

Let $\bm \mu = (\mu_1, \ldots, \mu_p)^\prime$ and $\bm X = (X_1, \ldots, X_p)^\prime.$ Following \cite{Dawid:1994}, we treat them as random vectors. We generalize \cite{Senn:2008} and assume that
\begin{equation}\label{eq:normal-normal}
\bm \mu \sim N(\bm 0, \bm \Sigma_0),
\quad
\bm X \mid \bm \mu \sim N(\bm \mu, \bm \Sigma),
\end{equation}
where
\begin{equation}\label{eq:nn-special}
\bm \Sigma_0 =
\gamma^2\bm I_p + (1-\gamma^2)\bm 1_p \bm 1_p^\prime,
\quad
\bm \Sigma =
\sigma^2\{ \eta^2\bm I_p + (1-\eta^2)\bm 1_p \bm 1_p^\prime \},
\quad
0 \le \gamma, \eta \le 1.
\end{equation}
To interpret \eqref{eq:nn-special} we let 
$
X_i = \mu_i + \epsilon_i,
$
where $\mu_i$ is generated by 
\begin{equation*}
\phi \sim N \left( 0, 1-\gamma^2 \right),
\quad
\mu_i \mid \phi \sim N \left( \phi, \gamma^2 \right),
\end{equation*}
and $\epsilon_i$ is generated by
\begin{equation*}
\xi \sim N \{ 0, (1-\eta^2)\sigma^2 \},
\quad
\epsilon_i \mid \xi \sim N(\xi, \eta^2 \sigma^2). 
\end{equation*}
Note that $\eta = 1$ in \cite{Senn:2008}, and we relax this assumption by allowing correlated errors.

\subsection{Posterior Mean}

To derive the posterior mean of $\mu_p$ given $X_1, \ldots, X_p,$ we rely on the following classic result.

\begin{lemma}[Normal Shrinkage]
\label{lemma:1}
Let
\begin{equation*}
\mu \sim N(\mu_0, \nu^2),
\quad
Z_i \mid \mu \stackrel{iid}{\sim} N(\mu, \tau^2) 
\quad
(i=1, \ldots, n).
\end{equation*}
Then the posterior mean of $\mu$ is
\begin{equation*}
\textrm E (\mu \mid Z_1, \ldots, Z_n) = \frac{\tau^2 \mu_0 + \nu^2 \sum_{i=1}^n Z_i}{\tau^2 + n\nu^2},    
\end{equation*}
\end{lemma}

\begin{prop}\label{prop:1}
The posterior mean of $\mu_p$ given $X_p$ is
\begin{eqnarray}\label{eq:post-mean-single}
\textrm E ( \mu_p \mid X_p ) & = & \frac{1}{1 + \sigma^2} X_p.
\end{eqnarray}
Furthermore, let
\begin{equation*}
a = \gamma^2 + \sigma^2\eta^2,
\quad
b = 1- \gamma^2 +  \sigma^2(1-\eta^2)
\end{equation*}
and
\begin{equation*}
r_1, \ldots, r_{p-1} = \frac{\sigma^2 (\eta^2 - \gamma^2)}{a(a+pb)},
\quad
r_p = \frac{a + (p-1)b\gamma^2}{a(a+pb)}.
\end{equation*}
The posterior mean of $\mu_p$ given $X_1, \ldots, X_p$ is
\begin{equation}\label{eq:post-mean}
\textrm E ( \mu_p \mid X_1, \ldots X_p ) = \sum_{i=1}^p r_i X_i.
\end{equation}
\end{prop}

\begin{proof}[Proof of Proposition \ref{prop:1}]
To prove the first half, notice that
\begin{equation*}
\mu_p \sim N(0, 1),
\quad
X_p \mid \mu_p \sim N(\mu_p, \sigma^2),
\end{equation*}
and apply Lemma \ref{lemma:1}.

To prove the second half, note that $\mu_i = \phi + \mu_i^\prime,$ where
\begin{equation*}
\phi \sim N \left( 0, 1-\gamma^2 \right),
\quad
\mu_i^\prime \stackrel{iid}{\sim} N \left( 0, \gamma^2 \right);
\end{equation*}
and $\epsilon_i = \xi + \epsilon_i^\prime,$ where
\begin{equation*}
\xi \sim N \{ 0, (1-\eta^2)\sigma^2 \},
\quad
\epsilon_i^\prime \stackrel{iid}{\sim} N(0, \eta^2 \sigma^2). 
\end{equation*}
Consequently we have
\begin{equation*}
\phi + \xi \sim N (0, b),
\quad
X_i \mid \phi + \xi \stackrel{iid}{\sim} N ( \phi + \xi, a ),
\end{equation*}
On the one hand, 
by Lemma \ref{lemma:1}
\begin{equation*}
\textrm E(\phi + \xi \mid X_1, \ldots X_p )
= \frac{b}{a+pb} \sum_{i=1}^p X_i,
\end{equation*}
and
\begin{equation*}
\textrm E(\phi \mid \phi + \xi, X_1, \ldots, X_p ) = \frac{1 - \gamma^2}{b} E( \phi + \xi \mid X_1, \ldots X_p ).
\end{equation*}
Consequently, 
\begin{eqnarray}\label{eq:alternative-2}
\textrm E(\phi \mid X_1, \ldots X_p ) 
& = & \textrm E \{ \textrm E(\phi \mid \phi + \xi, X_1, \ldots, X_p ) \mid X_1, \ldots X_p \} \nonumber \\
& = & \frac{1 - \gamma^2}{b} E( \phi + \xi \mid X_1, \ldots X_p ) \nonumber \\
& = & \frac{1 - \gamma^2}{a + pb} \sum_{i=1}^p X_i.
\end{eqnarray}
On the other hand, similarly we have
\begin{eqnarray}\label{eq:alternative-3}
\textrm E(\mu_p^\prime \mid X_1, \ldots X_p )
& = & \frac{\gamma^2}{a} E(\mu_i^\prime + \epsilon_i^\prime \mid X_1, \ldots X_p ) \nonumber \\
& = & \frac{\gamma^2}{a} 
\left\{
X_p - \frac{b}{a+pb} \sum_{i=1}^p X_i
\right\}.
\end{eqnarray}
Combine \eqref{eq:alternative-2} and \eqref{eq:alternative-3}, we complete the proof.
\end{proof}

It is worth noting that when $\gamma = \eta,$ \eqref{eq:post-mean} reduces to \eqref{eq:post-mean-single}.

\section{Closed-form expression for the Selection Bias}\label{sec:theory}

To simplify future notations, we assume that $X_p$ is the largest observation, i.e.,
$
X_p = \max_{1 \le i \le p}X_i.    
$
Consequently, the selection bias defined in \eqref{eq:target-original} becomes
\begin{equation}\label{eq:target}
\Delta = \textrm{E} ( \mu_p \mid X_p) - \textrm{E} ( \mu_p \mid X_p, X_p = \max X_i ).
\end{equation}
To derive its closed-form expression, we rely on the following lemmas. 

\begin{lemma}\label{lemma:2}
Let $\bm X_{-p} = (X_1, \ldots, X_{p-1})^\prime,$ and its distribution conditioning on $X_p$ is
\begin{equation}\label{eq:nn-cond}
N
\left(
\frac{b}{a+b} \bm 1_{p-1} X_p, 
\;
a \bm I_{p-1} + \frac{ab}{a+b} \bm 1_{p-1} \bm 1_{p-1}^\prime
\right).
\end{equation}
\end{lemma}

\begin{proof}[Proof of Lemma \ref{lemma:2}]
By \eqref{eq:normal-normal} we have $\bm X \sim N(0, \bm \Psi),$ where
\begin{equation*}
\bm \Psi = (\psi_{jk})_{1 \le j,k \le p} = a \bm I_p + b \bm 1_p \bm 1_p^\prime.
\end{equation*}
Furthermore, let
\begin{equation*}
\bm \Psi_{11} = (\psi_{jk})_{1 \le j,k \le p-1} = a \bm I_{p-1} + b \bm 1_{p-1} \bm 1_{p-1}^\prime,
\quad
\bm \Psi_{22} = (\psi_{pp}) = a + b,
\end{equation*}
and
\begin{equation*}
\bm \Psi_{12} = (\psi_{1p}, \ldots, \psi_{p-1,p})^\prime = b \bm 1_{p-1},
\quad
\bm \Psi_{21} = (\psi_{p1}, \ldots, \psi_{p, p-1}) = b \bm 1_{p-1}^\prime.
\end{equation*}
Simple probability argument suggests that
\begin{equation*}
\bm X_{-p} \mid X_p \sim 
N 
\left( 
\bm \Psi_{12}^{-1}\bm \Psi_{22}X_p,  
\bm \Psi_{11} - \bm \Psi_{12} \bm \Psi_{22}^{-1} \bm \Psi_{21}
\right),
\end{equation*}
where
\begin{equation*}
\bm \Psi_{12} \bm \Psi_{22}^{-1}X_p 
= \frac{b}{a + b} \bm 1_{p-1}X_p
\end{equation*}
and
\begin{eqnarray*}
\bm \Psi_{11} - \bm \Psi_{12} \bm \Psi_{22}^{-1} \bm \Psi_{21} 
& = & a \bm I_{p-1} + b \bm 1_{p-1} \bm 1_{p-1}^\prime - \frac{b^2}{a + b} \bm 1_{p-1} \bm 1_{p-1}^\prime \\
& = & a \bm I_{p-1} + \frac{ab}{a + b} \bm 1_{p-1} \bm 1_{p-1}^\prime
\end{eqnarray*}
The proof is complete.
\end{proof}

To state the next lemma, we introduce some notations. First, for $\bm \theta = (\theta_1, \ldots, \theta_n)^\prime$ and positive semi-definite matrix $\bm \Omega = (\omega_{jk})_{1\le j,k \le n},$ let
\begin{equation*}
\bm Y = (Y_1, \ldots, Y_n)^\prime \sim N(\bm \theta, \bm \Omega).
\end{equation*}
Second, let 
$
V_i = Y_i - \theta_i
$
for 
$
i = 1, \ldots, n. 
$
Consequently, 
\begin{equation*}
\bm V = (V_1, \ldots, V_n)^\prime \sim N(\bm 0, \bm \Omega),
\end{equation*}
whose probability density function is
\begin{equation*}
f(\bm v) = \frac{1}{(2\pi)^{n/2}|\bm \Omega|^{1/2}}\textrm{e}^
{
-\frac{1}{2} \bm v^\prime {\bm \Omega}^{-1} \bm v
},
\quad
\bm v = (v_1, \ldots, v_n)^\prime.
\end{equation*}
Third, for constants $b_1, \ldots, b_n,$ we let 
\begin{equation*}
\alpha 
= \mathrm{Pr} (V_1 \le b_1 - \theta_1, \ldots, V_n \le b_n - \theta_n)
=\int_{v_1 \le b_1 - \theta_1, \ldots, v_n \le b_n - \theta_n} f(\bm v)d\bm v,
\end{equation*}
and $\bm W = (W_1, \ldots, W_n)^\prime$ be the truncation version of $\bm V$ from above at $(b_1 - \theta_1, \ldots, b_n - \theta_n)^\prime.$ Consequently, its probability density function is
\begin{equation*}
g(\bm w) = \frac{1}{\alpha(2\pi)^{n/2}|\bm \Omega|^{1/2}}\textrm{e}^
{
-\frac{1}{2} \bm w^\prime {\bm \Omega}^{-1} \bm w
}
\cdot
1_{\{w_1 \le b_1 - \theta_1, \ldots, w_n \le b_n - \theta_n\}},
\quad
\bm w = (w_1, \ldots, w_n)^\prime.
\end{equation*}
For all $k=1, \ldots, n,$ let the $k$th marginal density function of $\bm W$ be
\begin{equation}\label{eq:tnorm-margin}
g_k(w) = \int_{-\infty}^{b_1 - \theta_1} \ldots \int_{-\infty}^{b_{k-1} - \theta_{k-1}} \int_{-\infty}^{b_{k+1} - \theta_{k+1}} \ldots \int_{-\infty}^{b_n - \theta_n} 
g(w_1, \ldots, w_{k-1}, w, w_{k+1}, \ldots, w_n)
\prod_{l \ne k}dw_l.
\end{equation}

For efficient analytical and numerical evaluations of \eqref{eq:tnorm-margin}, see \cite{Cartinhour:1990} and \cite{Wilhelm:2010}, respectively.

\begin{lemma}\label{lemma:3}
For all $i=1, \ldots, n,$
\begin{equation*}
\textrm E(Y_i \mid Y_1 \le b_1, \ldots, Y_n \le b_n) = \theta_i - \sum_{k=1}^n \omega_{ki} g_k(b_k - \theta_k).
 \end{equation*}
\end{lemma}

\begin{proof}[Proof of Lemma \ref{lemma:3}]
The proof follows \cite{Manjunath:2012}. First,
\begin{eqnarray}\label{eq:lemma-3-1}
\textrm E(Y_i \mid Y_1 \le b_1, \ldots, Y_n \le b_n) 
& = & \theta_i + E(V_i \mid V_1 \le b_1 - \theta_1, \ldots, V_n \le b_n - \theta_n) \nonumber \\
& = & \theta_i + E(W_i).
\end{eqnarray}
Next, the moment generating function of $\bm W$ at $\bm t = (t_1, \ldots, t_n)^\prime$ is 
\begin{eqnarray*}
m(\bm t) 
& = & \int \textrm e^{\bm t^\prime \bm w}g(\bm w) d\bm w \\
& = & \frac{1}{\alpha(2\pi)^{n/2}|\bm \Omega|^{1/2}} \int_{w_1 \le b_1 - \theta_1, \ldots, w_n \le b_n - \theta_n}
\textrm{e}^
{
-\frac{1}{2} 
\left(
\bm w^\prime {\bm \Omega}^{-1} \bm w - 2 \bm t^\prime \bm w
\right)
} 
d\bm w \\
& = & \underbrace{
\textrm e^{ \frac{1}{2} \bm t^\prime \bm \Omega \bm t }
\vphantom{ 
\frac{1}{\alpha(2\pi)^{n/2}|\bm \Omega|^{1/2}}
}
}_{m_1(\bm t)}
\underbrace{\frac{1}{\alpha(2\pi)^{n/2}|\bm \Omega|^{1/2}} \int_{w_1 \le b_1 - \theta_1, \ldots, w_n \le b_n - \theta_n}
\textrm e^
{
-\frac{1}{2} 
\left(
\bm w - \bm \Omega \bm t
\right)^\prime
{\bm \Omega}^{-1} 
\left(
\bm w - \bm \Omega \bm t
\right)
} 
d\bm w}_{m_2(\bm t)} \\
\end{eqnarray*}
On the one hand, by definition
\begin{eqnarray}\label{eq:lemma-3-2}
E(W_i) 
& = & \frac{\partial m (\bm t)}{\partial t_i} \rvert_{\bm t = \bm 0} \nonumber \\
& = & m_1(\bm 0) \frac{\partial m_2 (\bm t)}{\partial t_i} \rvert_{\bm t = \bm 0} + m_2(\bm 0) \frac{\partial m_1 (\bm t)}{\partial t_i} \rvert_{\bm t = \bm 0} \nonumber \\
& = & \frac{\partial m_2 (\bm t)}{\partial t_i} \rvert_{\bm t = \bm 0}.
\end{eqnarray}
On the other hand, let
\begin{equation*}
b^*_i = b_i - \theta_i - \sum_{k=1}^n \omega_{ik}t_k,
\quad
i = 1, \ldots, n,
\end{equation*}
and we can rewrite $m_2(\bm t)$ as
\begin{equation*}
m_2(\bm t) = \int_{-\infty}^{b_1^*} \ldots \int_{-\infty}^{b_n^*}
g(\bm w) dw_1 \ldots dw_n.
\end{equation*}
Therefore, by chain rule and Leibniz integral rule
\begin{eqnarray*}
\frac{\partial m_2 (\bm t)}{\partial t_i}
& = & \sum_{k=1}^n \frac{\partial b_k^*}{\partial t_i} \frac{\partial m_2 (\bm t)}{\partial b_k^*}\\
& = & - \sum_{k=1}^n \omega_{ki}
 \int_{-\infty}^{b_1^*} \ldots \int_{-\infty}^{b_{k-1}^*} \int_{-\infty}^{b_{k+1}^*} \ldots \int_{-\infty}^{b_n^*} g(w_1, \ldots, w_{k-1}, b_k^*, w_{k+1}, \ldots, w_n) \prod_{l \ne k}dw_l,
\end{eqnarray*}
and consequently
\begin{equation}\label{eq:lemma-3-3}
\frac{\partial m_2 (\bm t)}{\partial t_i} \rvert_{\bm t = 0}
= - \sum_{k=1}^n \omega_{ki} g_k(b_k - \theta_k).
\end{equation}
Combine \eqref{eq:lemma-3-1}, \eqref{eq:lemma-3-2} and \eqref{eq:lemma-3-3}, the proof is complete.
\end{proof}

\begin{prop}\label{prop:2}
For $i=1, \ldots, p-1,$ let $h_i$ denote the $i$th marginal probability density function of the random vector defined by \eqref{eq:nn-cond} truncated from above at $\bm 1_{p-1}X_p.$ Then the closed-form expression for \eqref{eq:target} is
\begin{equation}\label{eq:target-cf}
\Delta
= \frac{\sigma^2 (\eta^2 - \gamma^2)}{1 + \sigma^2} \sum_{i=1}^{p-1} h_i \left( \frac{\gamma^2 + \sigma^2 \eta^2}{1 + \sigma^2} X_p \right).
\end{equation}
\end{prop}

\begin{proof}[Proof of Proposition \ref{prop:2}]
Apply Lemma \ref{lemma:2} and \ref{lemma:3} to \eqref{eq:nn-cond},
\begin{equation*}
E (X_i \mid X_p, X_p = \max X_i ) = \frac{a}{a+b}X_p - 
\underbrace{
\left\{
\frac{ab}{a+b} \sum_{j=1}^{p-1} h_j \left( \frac{a}{a+b} X_p \right) + a h_i \left( \frac{a}{a+b} X_p \right)
\right\}
}_{\delta_i}.    
\end{equation*}
Consequently, by \eqref{eq:post-mean} we have
\begin{eqnarray*}
\textrm{E} ( \mu_p \mid X_p, X_p = \max X_i )
& = & r_p X_p + \sum_{i=1}^{p-1}r_i \textrm E (X_i \mid X_p, X_p = \max X_i ) \nonumber \\
& = & \left( 
r_p + \frac{a}{a+b} \sum_{i=1}^{p-1}r_i
\right)
X_p - \sum_{i=1}^{p-1} r_i \delta_i\nonumber \\
& = & \frac{X_p}{a+b} + 
\left\{
\frac{(p-1)ab}{a+b} + a
\right\}
\sum_{i=1}^{p-1} r_i h_i \left( \frac{a}{a+b} X_p \right) \nonumber \\
& = & \textrm E (\mu_p \mid X_p) - \frac{\sigma^2 (\eta^2 - \gamma^2)}{1 + \sigma^2} \sum_{i=1}^{p-1} h_i \left( \frac{\gamma^2 + \sigma^2\eta^2}{1+\sigma^2} X_p \right).
\end{eqnarray*}
The proof is complete.
\end{proof}

Proposition \ref{prop:2} confirms the existence of the selection bias in general. Furthermore, it provides the following interesting insights:
\begin{enumerate}
\item \label{intuition:1}
For fixed $\sigma,$ $p$ and $X_p,$ the sign of the selection bias is the same as the sign of
$
\eta^2 - \gamma^2,
$
i.e., depending on the correlation structures in \eqref{eq:nn-special}, neglecting the fact that $X_p = \max_{1 \le i \le p}X_i$ can either over-estimate or under-estimate $\mu_{i^*}.$ In particular, the selection bias is zero when $\gamma = \eta$. This is a generalization of the first main result in \cite{Senn:2008}, which assumes that $\gamma = \eta = 1;$

\item \label{intuition:2}
For fixed $\gamma$, $\eta$, $p$ and $X_p,$ the selection bias goes to zero as $\sigma$ goes to zero. This is intuitive because $X_p$ approaches $\mu_p$ as $\sigma$ goes to zero, and therefore the fact that $X_p = \max_{1 \le i \le p}X_i$ becomes irrelevant;

\item \label{intuition:3}
For fixed $\sigma,$ $\gamma$, $\eta$ and $p,$ the selection bias disappears for sufficiently large $X_p.$ This is because when $X_p$ goes to infinity,
\begin{equation*}
h_i \left( \frac{\sigma^2 + \gamma^2\eta^2}{1+\sigma^2}X_p \right) \rightarrow 0,
\quad
i=1, \ldots, p-1.
\end{equation*}
This result is in connection with \cite{Dawid:1973}.

\end{enumerate}

\section{Numerical and Simulated Examples}\label{sec:examples}

\subsection{Numerical Examples}

Having derived the closed-form expression for the selection bias, we provide some numerical examples for illustration. Let $\sigma =1,$ $p \in \{3, 5, 10\}$ and $X_p \in \{0, 1, \ldots, 6\}.$ For fixed $p$ and $X_p,$ we consider two cases. In Case 1, we follow \cite{Senn:2008} and let $\gamma^2=0.5$ and $\eta = 1.$ In Case 2, we let $\gamma=1$ and $\eta^2 = 0.5.$ For both cases we calculate the selection bias by \eqref{eq:target-cf}. Results are in Figure \ref{fg:nume}, which align with the insights discussed in the previous section. Furthermore, it appears that the magnitude of the selection bias increases as $p$ increases.

\begin{figure}[ht]
\centering
 \includegraphics[height=.5\linewidth, width= 1\linewidth]{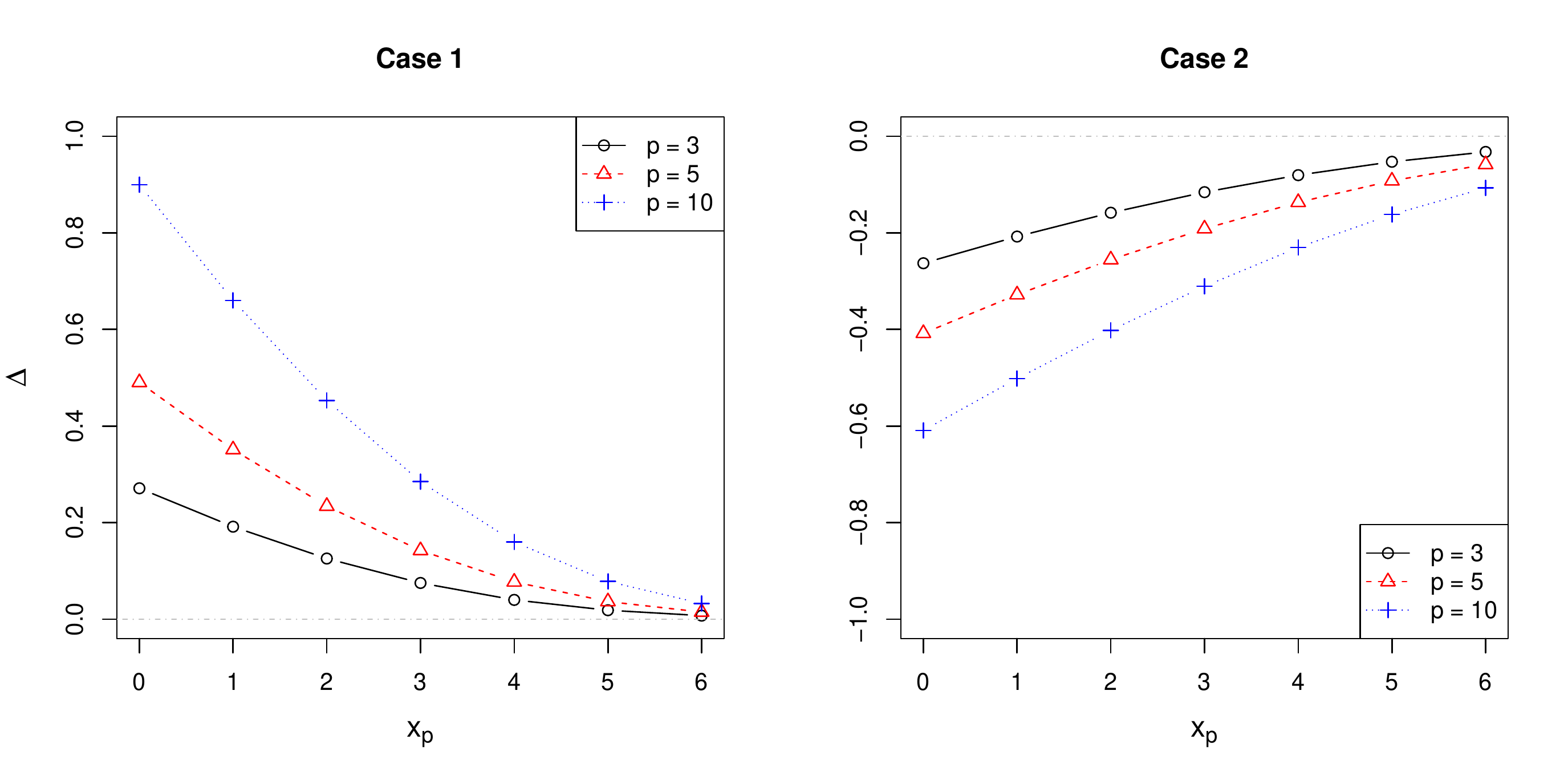}
\caption{Numerical Examples of Selection Bias.}
\label{fg:nume}
\end{figure}

\subsection{Simulated Examples}

The results in \eqref{eq:target-cf} enable us to calculate the ``exact post-selection'' posterior mean
\begin{equation}\label{eq:target-single}
\lambda_{i^*} = \textrm{E} ( \mu_{i^*} \mid X_{i^*}, X_{i^*} = \max X_i).
\end{equation}
For illustration, we revisit the simulated example in \cite{Senn:2008}, where $p=10,$ $\sigma=2,$ $\gamma^2=0.5$ and $\eta = 1.$ Figure \ref{fg:simu} contains 5000 pairs of $(\mu_{i^*}, X_{i^*})$ obtained by repeated sampling, the corresponding linear regression line that \cite{Senn:2008} used to approximate \eqref{eq:target-single}, and the curve that stands for the closed-form expression for \eqref{eq:target-single}. 

The results in Figure \ref{fg:simu} suggest that the regression approximation is relatively accurate for non-extreme values of  $X_{i^*}$ but not for extreme ones. Therefore our analytical solution has an advantage over the regression approximation in \cite{Senn:2008}. For further illustration we examine two concrete examples. First, let 
\begin{equation*}
x_{i^*} = 3.25,
\quad
\mathrm{Pr}(X_{i^*} > x_{i^*}) = 0.486.  
\end{equation*}
Therefore $3.25$ is a ``common'' value of $X_{i^*}.$ In this case the exact value of \eqref{eq:target-single} is
$
\lambda_{i^*} = 0.400
$
and the regression approximation is 
$
\hat \lambda_{i^*} = 0.368.
$
Consequently, although the ``absolute discrepancy'' $|\hat \lambda_{i^*} - \lambda_{i^*}| = 0.032$ seems small, the ``relative discrepancy''
\begin{equation*}
\frac{|\hat \lambda_{i^*} - \lambda_{i^*}|}{|\lambda_{i^*}|} = 8.1\%
\end{equation*}
is moderately large. Second, let 
\begin{equation*}
x_{i^*} = 1.5,
\quad
\mathrm{Pr}(X_{i^*} \le x_{i^*}) = 0.102.  
\end{equation*}
Therefore $1.5$ is a relatively ``uncommon'' (but not extreme) value of $X_{i^*}.$ In this case the absolute and relative discrepancies are respectively 0.062 and 24.7\%, both moderately large.

\begin{figure}[ht]
\centering
 \includegraphics[height=.5\linewidth, width= .5\linewidth]{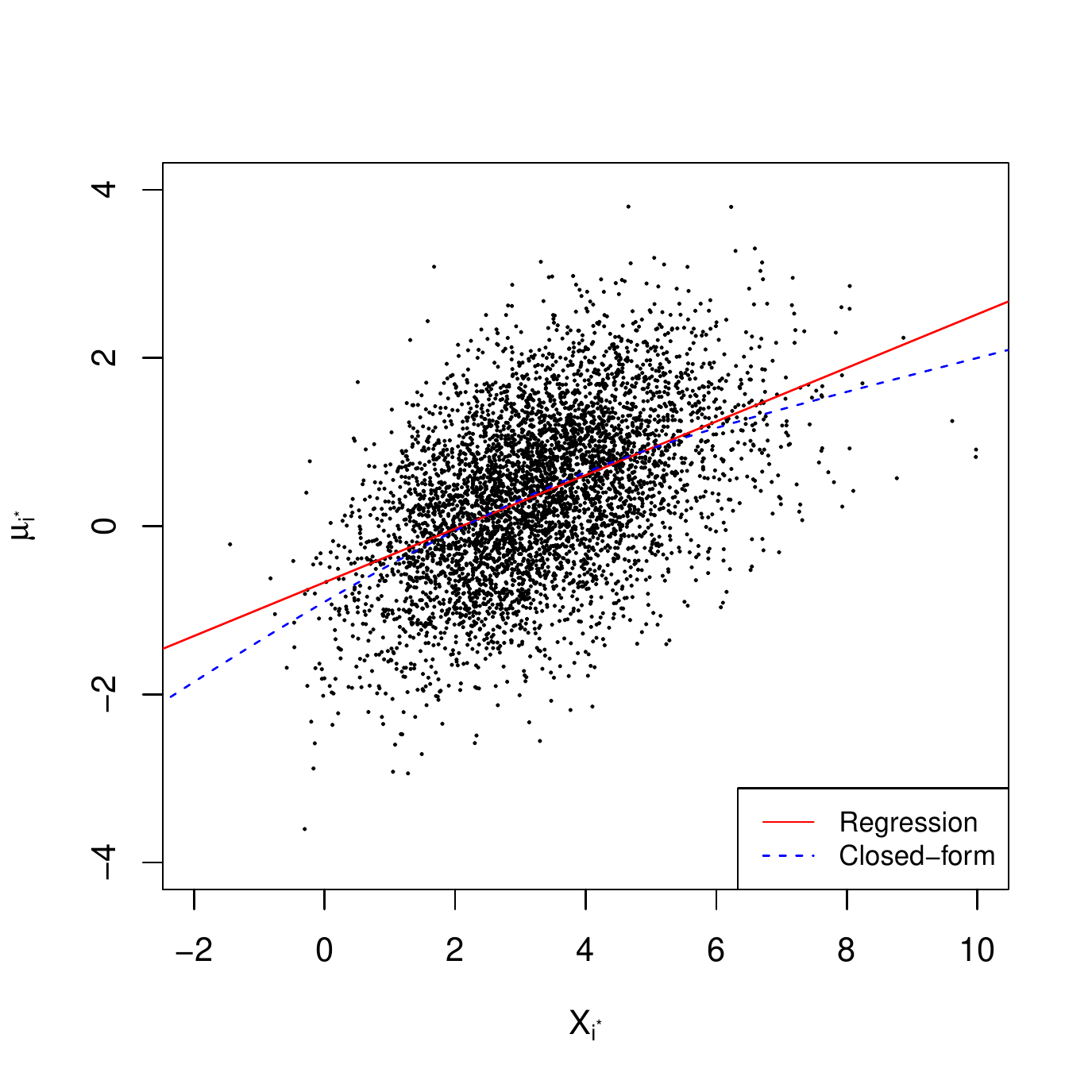}
\caption{``Exact Post-Selection'' Posterior Mean: Regression Approximation (Red Solid Line) and Closed-Form Expression (Blue Dotted Line).}
\label{fg:simu}
\end{figure}

\section{Concluding Remarks}\label{sec:discuss}

For the treatment selection problem, quantifying the selection bias is important from both theoretical and practical perspectives. In this paper, we extend the heuristic discussion in \cite{Senn:2008} and derive the closed-form expression for the selection bias. We illustrate the advantages of our results by numerical and simulated examples.

There are multiple possible future directions based on our current work. First, we can reconcile our Bayesian analysis with Frequentist methods. Second, it is possible to extend our results to more general model specifications by using the Tweedie's formula \citep{Robbins:1956, Efron:2011}. Third, we need to explore ``exact post-selection inference'' in multiple hypothesis testing.

\section*{Acknowledgements}

The authors thank the Co-Editor-in-Chief and a reviewer for their thoughtful comments that substantially improved the presentation of the paper.

\bibliographystyle{apalike}
\bibliography{selection_bias}

\end{document}